\documentclass[11pt]{amsart}

\usepackage{amsmath,amssymb,amsthm}
\usepackage{xcolor}
\usepackage{tikz}
\usetikzlibrary{arrows.meta,positioning}
\usepackage[colorlinks,linkcolor=blue,citecolor=blue]{hyperref}

\theoremstyle{plain}
\newtheorem{theorem}{Theorem}[section]
\newtheorem{proposition}[theorem]{Proposition}
\newtheorem{lemma}[theorem]{Lemma}
\newtheorem{corollary}[theorem]{Corollary}
\newtheorem*{thmA}{Theorem A}
\newtheorem*{thmB}{Theorem B}
\newtheorem*{thmC}{Theorem C}
\theoremstyle{definition}
\newtheorem{definition}[theorem]{Definition}

\theoremstyle{remark}
\newtheorem{remark}[theorem]{Remark}

\newcommand{\FF}{\mathcal{F}}
\newcommand{\GG}{\mathcal{G}}
\newcommand{\Ind}{\operatorname{Ind}}
\newcommand{\CS}{\operatorname{CS}}
\newcommand{\Sing}{\operatorname{Sing}}
\newcommand{\TTang}{\operatorname{Tang}}   
\newcommand{\tang}{\operatorname{tang}}    
\newcommand{\Res}{\operatorname{Res}}
\newcommand{\ord}{\operatorname{ord}}
\newcommand{\PP}{\mathbb{P}}
\newcommand{\CC}{\mathbb{C}}

\begin{document}

\title[Geometric index theorems for holomorphic foliations]{Geometric index theorems for holomorphic foliations and curves}

\author{C\'esar Camacho}
\address{School of applied mathematics  EMAp-FGV,  Rio de Janeiro,  Brazil}
\email{cesar.camacho@fgv.br}

\author{Rudy Rosas}
\address{Pontificia Universidad Cat\'olica del Per\'u, Lima, Peru}
\email{rudy.rosas@pucp.edu.pe}

\subjclass[2020]{32S65, 37F75, 53A20, 14H50}
\keywords{Holomorphic foliations, index theorems, projective connections,
Camacho--Sad index, tangencies, Pl\"ucker formula}

\begin{abstract}
Let $M$ be a complex surface endowed with a holomorphic projective
connection and let $C\subset M$ be a compact smooth holomorphic curve.
Given two singular holomorphic foliations $\FF$ and $\GG$ near $C$, with
$C$ not $\GG$-invariant, we attach to every point $p\in C$ an index
$\Ind(\FF,\GG,C,p)\in\CC$ of $\FF$ relative to the reference $\GG$, and we
prove that these indices add up to $T\GG\cdot C$. When $C$ is
$\FF$-invariant the index is the Camacho--Sad index corrected by the order
of tangency of $\GG$ with $C$, and the index formula reduces to the
Camacho--Sad formula.

The main applications are two contact formulas. At a tangency point $p$ of
$\GG$ with $C$, let $k(\GG,C,p)$ be the ratio of the curvatures at $p$ of
the leaf of $\GG$ and of the curve; it is a projective invariant, although
each curvature separately depends on a choice of metric. If $\GG$ has no
singular points on $C$ and only simple tangencies with $C$, then
\[
\sum_{p}\frac{1}{1-k(\GG,C,p)}\;=\;\frac23\,\bigl(C\cdot C+g-1\bigr),
\]
where $g$ is the genus of $C$: the number of tangencies depends on $\GG$,
but this weighted count does not. For a pencil of lines in the projective
plane it is the Pl\"ucker formula for the class of a plane curve. The
second formula asserts that, for a curve in general position with respect
to $\FF$ and $\GG$ which is not a geodesic, the total index of the
tangencies of $\FF$ with $\GG$ along $C$ equals the number of tangencies
of $\GG$ with $C$ minus one third of the number of inflection points of
$C$; in particular, it does not depend on $\FF$.
\end{abstract}

\maketitle
\section{Introduction}

Let $C$ be a compact smooth holomorphic curve in a complex surface $M$ and
let $\FF$ be a singular holomorphic foliation defined near $C$, not
necessarily leaving $C$ invariant. Let $\GG$ be a second singular
holomorphic foliation, defined near $C$ and not leaving $C$ invariant,
which plays the role of a \emph{reference}: the foliation $\FF$ is measured
against it along $C$.
Under the assumption that $M$ carries a holomorphic projective connection
--- the projective plane being the basic example --- we attach to every
point $p\in C$ an index
\[
\Ind(\FF,\GG,C,p)\in\CC ,
\]
which vanishes at all but finitely many points, and we prove
(Theorem~A) that
\[
\sum_{p\in C}\Ind(\FF,\GG,C,p)\;=\;T\GG\cdot C .
\]

When $C$ happens to be $\FF$-invariant, the index is the Camacho--Sad index
$\CS(\FF,C,p)$ of \cite{CS}, corrected by the order of tangency of the
reference with the curve,
\[
\Ind(\FF,\GG,C,p)\;=\;\CS(\FF,C,p)\;-\;\TTang(\GG,C,p)
\]
(Proposition~\ref{prop:invariant}); in particular the index does not
depend on the projective connection in this case. Since the tangency orders
of $\GG$ with $C$ add up to $C\cdot C-T\GG\cdot C$, Theorem~A becomes in
this case the Camacho--Sad formula
\[
\sum_{p}\CS(\FF,C,p)\;=\;C\cdot C ,
\]
and is thus a generalization of it to surfaces endowed with a projective
connection.

The index $\CS(\FF,C,p)$ has been generalized in other directions, to
non-invariant curves and to higher dimensions \cite{Ca,CMS,CL,ABT}, and
it is worth explaining where our theorem stands in relation to these
results. In those works the index requires, as an additional reference, a
transverse structure defined along the whole curve and the
existence of such a structure depends on how $C$ is embedded in the
ambient space. In $\PP^2$, the archetypal surface with a projective
structure, a compact smooth curve admits such a structure only when it is a line.
Our index, by contrast, takes as reference an
arbitrary second foliation $\GG$, not everywhere tangent to $C$, and
uses the projective structure to define in a natural way the index
$\Ind(\FF,\GG,C,p)$, attached to the two foliations and to any smooth
curve. We emphasize that this index is a \emph{projective} invariant, not
an analytic one: it is preserved by biholomorphisms which respect the
projective connections. What the index measures lives in the projective
geometry of $M$.

\subsection{The setting}
The role of the projective connection is the following. Along a curve
which is not invariant, the index has to be extracted from the derivative
of the slope of $\FF$, relative to $\GG$, in the direction transverse to
$C$; and this derivative is sensitive to the second-order jet of the
transverse parameter used to compute it. Since we want the index to be
defined for \emph{every} curve in the surface, \emph{every} foliation and
\emph{every} reference at once, the normalization of transverse parameters
cannot be chosen curve by curve: it has to come from the ambient surface.
This uniformity is a strong demand, and it constrains the surface. A
holomorphic projective connection supplies exactly what is needed: its
geodesics provide, canonically and simultaneously for all curves and all
references, transverse parameters normalized to second order. It is the
natural ambient structure for this index, and it is the setting of the
present paper.

A projective connection on $M$ is an atlas of holomorphic torsion-free
connections $\{(U_i,\nabla_i)\}$ whose members are projectively equivalent
--- in Weyl's classical sense
\cite{We} --- on overlaps (Section~\ref{sec:preliminaries}). The
\emph{geodesics} of the projective connection, the holomorphic curves
satisfying $\nabla_{\dot c}\dot c\wedge\dot c\equiv0$, are then globally
well defined: through every point and every direction passes exactly one.
The basic example is the \emph{flat} case, the holomorphic projective
structures of Ehresmann--Thurston: an atlas of charts with values in
$\PP^2$, with transitions in $\mathrm{PGL}(3,\CC)$, whose distinguished
charts carry the pullbacks of the affine connection of
$\CC^2\subset\PP^2$ and whose geodesics are the preimages of the projective
lines. The projective plane carries its tautological flat projective
connection, so all our results apply to foliations of $\PP^2$, where they
are already new; but the theorems are proved for an arbitrary projective
connection.

We emphasize two features of the setting. First, the surface $M$ is
\emph{not} assumed to be compact: a neighborhood of the curve $C$ is enough.
This matters, because compact complex surfaces carrying a flat projective
connection are completely classified \cite{KO} and form a short list. Second, the reference
$\GG$ is an arbitrary singular foliation: it is \emph{not} required to be
transverse to $C$, and isolated tangencies of $\GG$ with $C$, and even
singularities of $\GG$ on $C$, are allowed. Far from being a nuisance, the
tangencies of the reference with the curve are the source of the
applications: each of them carries a contact invariant, defined by the
projective connection, which computes the index there whenever $\FF$ is
transversal to $C$ at the point; and the freedom in the choice of the
reference makes the reference itself an object of study. The curve $C$ may
be invariant by $\FF$,
and $\FF$ may have singularities on $C$. Only two configurations are
excluded: $C$ invariant by $\GG$, and $\FF$ tangent to $\GG$ along the whole
of $C$.

\subsection{The main theorem}
For a foliation $\GG$ we denote by $T\GG$ its tangent bundle, defined by
local generators with isolated zeros; for a line bundle $E$ on a
neighborhood of $C$ we write $E\cdot C:=\deg\,(E|_C)$.

\begin{thmA}
Let $M$ be a complex surface endowed with a projective connection and let
$C\subset M$ be a compact smooth holomorphic curve. Let $\GG$ be a singular
holomorphic foliation defined in a neighborhood of $C$, such that $C$ is not
$\GG$-invariant, and let $\FF$ be a singular holomorphic foliation defined in
a neighborhood of $C\setminus\Sigma$, where $\Sigma\subset C$ is finite, such
that $\TTang(\FF,\GG)\cap C$ is a finite set. Then
$\Ind(\FF,\GG,C,p)$ vanishes for all but finitely many $p\in C$, and
\[
\sum_{p\in C}\Ind(\FF,\GG,C,p)\;=\;T\GG\cdot C .
\]
\end{thmA}

The index can be nonzero only at three kinds of points: the tangencies of
$\GG$ with $C$ (including the singular points of $\GG$ on $C$), the points of
$\Sigma$, and the tangencies of $\FF$ with $\GG$ along $C$ (including the
singular points of $\FF$ on $C$).

A feature of the statement deserves emphasis, since it is used in an
essential way in the proof of the contact formula below: the foliation $\FF$
is only required to be defined in a neighborhood of $C\setminus\Sigma$, and
\emph{no extension of $\FF$ across the points of $\Sigma$ is required}. 

In fact the index depends on the auxiliary foliation $\GG$ only through the
restriction $T\GG|_C$ --- more precisely, through the morphism
$(T\GG\to TM)|_C$. This leads to the natural generality in which the theorem
will be proved: the auxiliary datum is a morphism
\[
\psi\colon L\longrightarrow TM|_C
\]
from a line bundle $L$ on $C$, not everywhere tangent to $C$, and the index
formula reads
\[
\sum_{p\in C}\Ind(\FF,\psi,C,p)\;=\;c_1(L).
\]
An important particular case is that of a line \emph{subbundle}
$V\subset TM|_C$, for which we write $\Ind(\FF,V,C,p)$.

\subsection{Two contact formulas}
The applications of Theorem~A given in this paper are two global formulas
for the contact invariants attached to tangency points.

We first recall a classical count. Let $C\subset M$ be a compact smooth
curve. The projective second fundamental form of $C$ (Section
\ref{sec:preliminaries}) is a canonical section $\sigma$ of
$N_C\otimes (T^*C)^{\otimes2}$; its zeros are the \emph{inflection points} of
$C$, the points where the tangent geodesic has contact of order $\ge 3$ with
$C$. If $C$ is not a geodesic then $\sigma\not\equiv0$ and the number of
inflection points, counted with multiplicity, is
\[
\iota(C)\;=\;C\cdot C+4\,(g-1),
\]
where $g$ is the genus of $C$. For a smooth plane curve of degree $c$ this is
the classical count $3c(c-2)$.

At a point $p$ where two smooth holomorphic curves $C_1,C_2$ are tangent we
attach a \emph{contact invariant} $k(C_1,C_2,p)\in\PP^1$: since the two
curves have the same tangent line at $p$,
their projective second fundamental forms at $p$ are two elements of the same complex line $N_p\otimes(T^*_pC)^{\otimes2}$, and we set
\[
k(C_1,C_2,p)\;:=\;\frac{\sigma_{C_1}(p)}{\sigma_{C_2}(p)}\;\in\;\PP^1 .
\]
The invariant $k(C_1,C_2,p)$ is a genuine \emph{ratio of curvatures}: If there exists a smooth local hermitian metric compatible with the projective connection, the corresponding
second fundamental forms at $p$ are two vectors on the same one-dimensional linear space, and $k(C_1,C_2,p)$ is their quotient (Remark~\ref{rem:kvalues}).
Individually, the two curvatures depend on the metric; their quotient does
not.

In local holomorphic coordinates in which $p=(0,0)$ and  $\{y=0\}$ is a geodesic 
 and is the common tangent, writing $C_i$ as the
graph of $y=f_i(x)$ one has $k(C_1,C_2,p)=f_1''(0)/f_2''(0)$
(Remark~\ref{rem:kvalues}(d)). The tangency of $C_1$ and $C_2$
at $p$ is \emph{simple} --- of contact order exactly two --- precisely when
$k(C_1,C_2,p)\neq1$; the value $k(C_1,C_2,p)=\infty$, which occurs when $p$
is an inflection point of $C_2$, is allowed, with the convention
$\tfrac1{1-\infty}=0$.

If $\GG$ is a foliation tangent to $C$ at $p$, and $\mathcal L$ denotes the
leaf of $\GG$ through $p$, we write
\[
k(\GG,C,p)\;:=\;k(\mathcal L,C,p) .
\]

\begin{thmB}[A contact formula for foliations and curves]
Let $M$ be a complex surface endowed with a projective connection, let
$C\subset M$ be a compact smooth holomorphic curve of genus $g$, and let
$\GG$ be a singular holomorphic foliation defined on a neighborhood of $C$,
with no singular point on $C$, not leaving $C$ invariant, and having only
simple tangencies with $C$, at the points $p_1,\dots,p_n$. Then
\[
\sum_{j=1}^{n}\frac{1}{1-k(\GG,C,p_j)}\;=\;\frac23\,\bigl(C\cdot C+g-1\bigr).
\]
\end{thmB}

The right hand side does not depend on $\GG$: this is the content of the
formula. The \emph{number} $n$ of tangencies does depend on $\GG$ --- it
equals $C\cdot C-T\GG\cdot C$ --- but the weighted sum does not.

Theorem~B is an extension of the Pl\"ucker formula. Suppose $M=\PP^2$ and
take $C$ smooth of degree $c$ and $\GG$ of degree $d$; then $n=cd+c(c-1)$
and the sum equals $c(c-1)$, the \emph{class} of $C$. If $\GG$ is a pencil
of lines through a generic point, every leaf is a line, every invariant $k$
vanishes, every weight equals $1$, and the formula reduces to the classical
count of the tangent lines to $C$ from a point (see e.g.\ \cite{GHP}).
Replacing the pencil by an arbitrary foliation changes the number of
tangency points, but the total count remains constant, provided each
tangency is reweighted by the factor $1/(1-k)$.

In Section~\ref{sec:contact1} we prove a more general statement, for line
subbundles of $TM|_C$ instead of foliations.

The second application concerns the tangencies of two foliations along a
curve. We say that a compact smooth curve $C$ is in \emph{general position}
with respect to the pair $(\FF,\GG)$ if $\GG$ has no singular point on $C$
and $C$ is not $\GG$-invariant, if all tangencies of $\GG$ with $C$ are
simple and $\FF$ is transversal to $C$ at each of them, and if $\FF$ and
$\GG$ are not tangent along the whole of $C$
(Section~\ref{sec:contact2}).

\begin{thmC}[A contact formula for two foliations along a generic curve]
Let $M$ be a complex surface endowed with a projective connection, let
$\FF,\GG$ be singular holomorphic foliations on a neighborhood of a compact
smooth curve $C$, and assume $C$ is in general position with respect to
$(\FF,\GG)$. Then
\[
\sum_{p\,\in\,\TTang(\FF,\GG)\cap C}\Ind(\FF,\GG,C,p)
\;=\;\tang(\GG,C)\;-\;\frac{\iota(C)}{3}\,,
\]
where $\tang(\GG,C)$ is the number of tangencies of $\GG$ with $C$ and
$\iota(C)=C\cdot C+4(g-1)$.
\end{thmC}

The right hand side does not depend on $\FF$. When $C$ is not a geodesic,
$\iota(C)$ is the number of inflection points of $C$, and the formula
reads: whatever the foliation $\FF$, the total index of its tangencies with
$\GG$ along $C$ is the number of tangencies of the \emph{reference}
foliation with $C$, minus one third of the number of inflection points of
the curve. In the projective plane the right hand side equals
$\deg C\,(\deg\GG+1)$.

\subsection{Organization of the paper}
Section~\ref{sec:preliminaries} collects preliminaries on projective
connections, geodesics and inflections of curves.
Section~\ref{sec:extension} contains the extension lemma, which produces the
canonical second-order normalization mentioned above.
Section~\ref{sec:index} defines the index and establishes its invariance
properties. Section~\ref{sec:main} proves Theorem~A and derives the
Camacho--Sad formula. Section~\ref{sec:local} computes the index in the
four basic local configurations. Sections~\ref{sec:contact1} and
\ref{sec:contact2} prove Theorems~B and~C.

\section{Preliminaries}\label{sec:preliminaries}

\subsection{Connections and projective equivalence}\label{subsec:compatible}
Let $M$ be a complex surface. By a \emph{connection} on an open set
$U\subset M$ we mean a holomorphic torsion-free connection on $TU$: in a
holomorphic chart it is given by holomorphic Christoffel symbols
$\Gamma^k_{ij}$, symmetric in $i,j$.

Two connections $\nabla,\nabla'$ on $U$ are \emph{projectively equivalent}
if $$\nabla'_XY=\nabla_XY+\xi(X)\,Y+\xi(Y)\,X $$ for some holomorphic $1$-form $\xi$ on $U$. The equation 
 $\gamma'\wedge \nabla_{\gamma'}\gamma'=0$, whose solutions are the geodesics up to reparametrization, remains invariant under the change of $\nabla$ by projectively equivalent connection
  $\nabla'$.   Projective equivalence is
exactly the condition of having the same geodesics as unparametrized
curves.

\subsection{Projective connections}\label{subsec:projconn}

A \emph{projective connection} on $M$ is an atlas of connections
$\{(U_i,\nabla_i)\}$, where $\{U_i\}$ is an open cover of $M$, such that
$\nabla_i$ and $\nabla_j$ are projectively equivalent on $U_i\cap U_j$
whenever $U_i\cap U_j\neq\emptyset$. Two atlases define the same projective
connection if their union is again an atlas as above. We denote a projective
connection by $[\nabla]$ and call the $\nabla_i$  its \emph{local
representatives}.
A \emph{geodesic} of a projective connection $[\nabla]$ is a holomorphic
curve $c$ with
$\dot c\wedge\nabla_{\dot c}\dot c\equiv0$ for one --- hence, for every --- local representative $\nabla$.
Through every point of $M$ and every tangent direction there passes exactly
one geodesic, and it depends holomorphically on its $1$-jet.

\subsection{Flat projective connections}\label{subsec:flat}
A \emph{holomorphic projective structure} on $M$ (in the sense of
Ehresmann--Thurston) is a holomorphic atlas $\{(U_i,\varphi_i)\}$,
$\varphi_i\colon U_i\to\PP^2$, whose transition maps
$\varphi_i\circ\varphi_j^{-1}$ are restrictions of elements of
$\mathrm{PGL}(3,\CC)$. The charts $\varphi_i$ are called
\emph{distinguished charts}.  If we choose the charts $\varphi$ taking values in $\mathbb C^2$,
we can consider in each $U_i$ the pullback $\varphi_i^*(d)$ of the flat connection $d$ in $\mathbb C^2$. This atlas defines a projective connection on $M$; we
call it \emph{flat}, and we call  $\varphi_i^*(d)$ the \emph{flat
representative} of the chart. Its geodesics are the curves mapped by the
distinguished charts into projective lines.

The projective plane carries its tautological flat projective connection.
A complex torus $\CC^2/\Lambda$ carries the one induced by the affine
structure of $\CC^2$, whose geodesics are the images of the affine lines.
Quotients of the complex $2$-ball by discrete groups of automorphisms carry
the one induced by the inclusion $\mathbb{B}^2\subset\PP^2$. For compact
$M$ the surfaces carrying a projective structure are classified \cite{KO}.

\subsection{Inflections of a curve}\label{subsec:inflections}
Let $C\subset M$ be a smooth holomorphic curve and let $\nabla$ be a local
representative of $[\nabla]$ near a point of $C$. For local sections $X,Y$
of $TC$ put
\[
\sigma(X,Y)\;:=\;\nabla_XY \ \ \mathrm{mod}\ TC\;\in\;N_C .
\]

\begin{lemma}\label{lem:sff}
$\sigma$ is $\mathcal O_C$-bilinear and symmetric, and does not depend on
the choice of the representative. It therefore defines a canonical global
section
\[
\sigma\;\in\;H^0\!\bigl(C,\;N_C\otimes (T^*C)^{\otimes2}\bigr),
\]
the \emph{projective second fundamental form} of $C$.
\end{lemma}

The bilinearity and symmetry of $\sigma$ for a fixed ambient connection are
classical: $\sigma$ is the second fundamental form of $C$ relative to
$\nabla$, in the sense of affine differential geometry \cite{NS}. The only
point to note is the projective invariance, which is immediate: if $\nabla$
is replaced by a projectively equivalent connection, the extra term
$\xi(X)Y+\xi(Y)X$ lies in $TC$ whenever $X,Y\in TC$, so $\sigma$ is
unchanged.

A point $p\in C$ is an \emph{inflection point} of $C$ if $\sigma(p)=0$;
equivalently, if the geodesic tangent to $C$ at $p$ has contact of order
$\geq3$ with $C$ at $p$. The curve $C$ is a geodesic if
and only if $\sigma\equiv0$.

\begin{proposition}\label{prop:inflections}
Let $C\subset M$ be a compact smooth holomorphic curve of genus $g$, not a
geodesic. Then the number of inflection points of $C$, counted with
multiplicity, is
\[
\iota(C)\;=\;\deg\operatorname{div}(\sigma)\;=\;C\cdot C+4\,(g-1).
\]
\end{proposition}

\begin{proof}
Since $C$ is not a geodesic, $\sigma\not\equiv0$ and
$\iota(C)=\deg\bigl(N_C\otimes(T^*C)^{\otimes2}\bigr)=C\cdot C+2\,(2g-2)$.
\end{proof}

For a smooth plane curve of degree $c$ one gets
$c^2+4\bigl(\tfrac{(c-1)(c-2)}2-1\bigr)=3c(c-2)$, the classical count of
inflection points. For a geodesic
curve the number $\iota(C)=C\cdot C+4(g-1)$ is still defined but no longer
counts anything, since $\sigma\equiv0$.

\section{The extension lemma}\label{sec:extension}

Throughout this section $U\subset\CC$ is an open set, $\gamma\colon U\to M$
a holomorphic embedding, and $C:=\gamma(U)$ its image, a smooth holomorphic
curve in $M$. Let $Z$ be a holomorphic vector field
along $C$ with $Z\wedge\gamma'\not\equiv0$. We denote by
\[
T_Z\;:=\;\{\,Z\wedge\gamma'=0\,\}\;\subset\;C
\]
the \emph{degeneracy set} of $Z$; it is discrete by hypothesis, and consists
of the points where $Z$ is tangent to $C$ \emph{together with the zeros of
$Z$}; the multiplicity of $p\in T_Z$ is
$\ord_p(Z\wedge\gamma')$. 

Fix a local representative $\nabla$ of the projective connection and set
\begin{equation}\label{eq:h}
h\;=\;h_Z\;:=\;\frac{Z\wedge\nabla_{\gamma'}Z}{Z\wedge\gamma'} ,
\end{equation}
a holomorphic function on $C\setminus T_Z$. It does not depend on the
parametrization $\gamma$, both numerator and denominator being linear in
$\gamma'$. 

\begin{definition}\label{def:permitted}
A \emph{permitted extension} of $Z$ (relative to $\nabla$) is a holomorphic
vector field $\tilde Z$ on a neighborhood of $C\setminus T_Z$ in $M$ such
that
\[
\tilde Z\big|_{C}\;=\;Z,
\qquad
\bigl(\nabla_{\tilde Z}\tilde Z\bigr)\big|_{C}\;=\;2\,h\,Z .
\]
\end{definition}

\begin{lemma}[Extension lemma]\label{lem:extension}Let $Z$ and $h$  be as above.
\begin{enumerate}
\item[(i)] The class of permitted extensions depends only on the
projective connection. If
$\nabla'$ is projectively equivalent to $\nabla$,  a field is permitted relative to $\nabla$ if and only if
it is permitted relative to $\nabla'$. 
\item[(ii)] Permitted extensions exist. Put
$U_Z:=U\setminus\gamma^{-1}(T_Z)$. Let $\nabla$ be a local representative
and, for $s\in U_Z$, let $c_s$ be the parametrized geodesic relative to the connection $\nabla$ with
$c_s(0)=\gamma(s)$ and $\dot c_s(0)=Z(s)$. Then
\begin{equation}\label{eq:phi}
\phi(s,t)\;:=\;c_s\bigl(u(s,t)\bigr),
\qquad
u(s,t)\;:=\;\frac{t}{1-h(s)\,t},
\end{equation}
is defined and holomorphic on a neighborhood $\Omega$ of $U_Z\times\{0\}$
in $\CC^2$. After shrinking $\Omega$, the map $\phi$ is a biholomorphism
onto a neighborhood of $\gamma(U_Z)$ in $M$, its curves
$s=\mathrm{const}$ are geodesics, and $\tilde Z:=\phi_*\partial_t$ is a
permitted extension of $Z$. Note that $\phi(\cdot,0)=\gamma$, so that the
line $\{t=0\}$ is carried onto the curve by the parametrization.
\end{enumerate}
\end{lemma}

\begin{proof}
(i) Let $h'=\frac{Z\wedge\nabla'_{\gamma'}Z}{Z\wedge\gamma'}$.  One has $\nabla'_{\gamma'}Z=\nabla_{\gamma'}Z
+\xi(\gamma')Z+\xi(Z)\gamma'$, and wedging with $Z$ kills the first extra
term: $h'=h+\xi(Z)$. On the other hand, along $C$ one has $\tilde Z=Z$, so
$$\nabla'_{\tilde Z}\tilde Z|_C=\nabla_{\tilde Z}\tilde Z|_C+2\,\xi(Z)\,Z=2hZ+2\,\xi(Z)\,Z=2h'Z.$$

(ii) The geodesic $c_s(u)$ depends holomorphically on
$(s,u)$ and is defined for $(s,u)$ in a neighborhood of $U_Z\times\{0\}$;
as $u(s,0)=0$, the map $\phi$ is defined and holomorphic near
$U_Z\times\{0\}$. From
\[
u(s,0)=0,\qquad u_t(s,0)=1,\qquad u_{tt}(s,0)=2h(s),
\]
we get $\phi(s,0)=\gamma(s)$ and $\partial_t\phi(s,0)=Z(s)$, so that the
differential of $\phi$ at $(s,0)$ sends $(\partial_s,\partial_t)$ to
$(\gamma'(s),Z(s))$, which are linearly independent precisely for
$s\in U_Z$. Hence $\phi$ is a local biholomorphism at each point of
$U_Z\times\{0\}$; since $\phi(\cdot,0)=\gamma$ is an embedding, $\phi$ is
injective on a suitable neighborhood of $U_Z\times\{0\}$, and being open it
maps that neighborhood biholomorphically onto a neighborhood of
$\gamma(U_Z)$. The curves $s=\mathrm{const}$ are geodesics by construction.

Finally, since
$\tilde Z=\phi_*\partial_t=u_t\,\dot c_s(u)$,

\[
\nabla_{\tilde Z}\tilde Z
\;=\;\frac{D}{dt}\bigl(u_t\,\dot c_s(u)\bigr)
\;=\;u_{tt}\,\dot c_s(u)+u_t^{\,2}\,\bigl(\nabla_{\dot c}\dot c\bigr)(u)
\;=\;u_{tt}\,\dot c_s(u),
\]
which at $t=0$ equals $2h(s)Z(s)$. Thus $\tilde Z$ is permitted.
\end{proof}

The construction of (ii) attaches a system of coordinates to one particular
permitted extension. The same can be done for any permitted extension. Let
$\tilde Z$ be an arbitrary permitted extension and let $\phi(s,t)$ be the point reached at
time $t$ by the integral curve of $\tilde Z$ issued from $\gamma(s)$. Then
$\phi$ is defined and holomorphic on a neighborhood of $U_Z\times\{0\}$,
$\phi(\cdot,0)=\gamma$, and $d\phi_{(s,0)}$ sends $(\partial_s,\partial_t)$
to $(\gamma'(s),Z(s))$, which are linearly independent precisely for
$s\in U_Z$; the argument of (ii) then shows that, after shrinking, $\phi$ is
a biholomorphism onto a neighborhood of $\gamma(U_Z)$ in $M$. We call
$(s,t)$ the \emph{adapted coordinates} of $\tilde Z$; by construction
$\tilde Z=\phi_*\partial_t$. For the extension built in (ii) they are the
coordinates given by \eqref{eq:phi}.

\begin{figure}[ht]
\centering
\begin{tikzpicture}[scale=0.95,>=Stealth,line join=round]

\begin{scope}
  \fill[black!6] (0,0) -- (2.4,1.08) -- (2.4,-1.08) -- cycle;
  \fill[black!6] (0,0) -- (-2.4,1.08) -- (-2.4,-1.08) -- cycle;

  \draw[->] (-3,0) -- (3.15,0) node[right] {$s$};
  \draw[->] (0,-1.65) -- (0,1.65) node[above] {$t$};

  \draw[black!55] (-2.55,-1.15) -- (2.55,1.15);
  \draw[black!55] (-2.55,1.15) -- (2.55,-1.15);

  \foreach \s in {-2.2,-1.8,-1.4,-1.0,-0.6,0.6,1.0,1.4,1.8,2.2}
     \draw[black!70] ({\s},{0.45*\s}) -- ({\s},{-0.45*\s});

  \draw[very thick] (-2.4,0) -- (2.4,0);
  \node[above] at (-1.9,0.06) {$\{t=0\}$};

  \fill[white] (0,0) circle (3.4pt);
  \draw (0,0) circle (3.4pt);
  \draw (-0.085,-0.085) -- (0.085,0.085);
  \draw (-0.085,0.085) -- (0.085,-0.085);
  \node[below left=2pt and 0pt] at (0,-0.06) {$s_0$};

  \node at (0.15,-2.15) {$\Omega\subset\CC^{2}$};
\end{scope}

\draw[->,thick] (3.45,1.15) to[bend left=18] (5.15,1.15);
\node at (4.3,1.85) {$\phi$};

\begin{scope}[xshift=8cm]
  \fill[black!6]
     plot[domain=-2:2,samples=80,variable=\s] ({1.45*\s},{0.345*\s*\s})
  -- plot[domain=2:-2,samples=80,variable=\s] ({0.55*\s},{-0.105*\s*\s})
  -- cycle;

  \draw[black!55]
     plot[domain=-2:2,samples=80,variable=\s] ({1.45*\s},{0.345*\s*\s});
  \draw[black!55]
     plot[domain=-2:2,samples=80,variable=\s] ({0.55*\s},{-0.105*\s*\s});

  \foreach \s in {-2,-1.7,-1.4,-1.1,-0.8,-0.5,0.5,0.8,1.1,1.4,1.7,2}
     \draw[black!70] ({1.45*\s},{0.345*\s*\s}) -- ({0.55*\s},{-0.105*\s*\s});

  \draw[very thick]
     plot[domain=-2.75:2.75,samples=100,variable=\s] ({\s},{0.12*\s*\s});
  \node[right] at (2.75,0.91) {$C$};

  \fill (1.7,0.347) circle (1.5pt);
  \node[above left=0pt and 1pt] at (1.7,0.36) {$\gamma(s)$};

  \fill[white] (0,0) circle (3.4pt);
  \draw (0,0) circle (3.4pt);
  \draw (-0.085,-0.085) -- (0.085,0.085);
  \draw (-0.085,0.085) -- (0.085,-0.085);
  \node[below=3pt] at (0,-0.06) {$p$};

  \node at (0.15,-2.15) {$p\in T_Z$};
\end{scope}

\end{tikzpicture}
\caption{The parametrization $\phi$ near a degeneracy point
$p=\gamma(s_0)\in T_Z$. }
\label{fig:phi}
\end{figure}
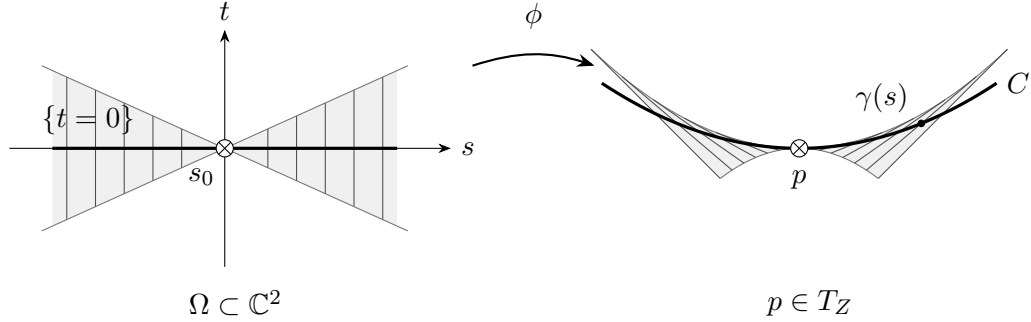

\begin{remark}[Rescaling]\label{rem:rescaling-h}
For a holomorphic nonvanishing function $\theta$ on $C$ one has
\begin{equation}\label{eq:htheta}
h_{\theta Z}\;=\;\theta\,h_Z.
\end{equation}
Moreover the
geodesic with initial velocity $\theta(s)Z(s)$ is the geodesic $c_s$
traversed at $\theta(s)$ times the speed,
\[
c_s^{\,\theta Z}(u)\;=\;c_s^{\,Z}\bigl(\theta(s)\,u\bigr).
\]
Combining this with \eqref{eq:htheta} in \eqref{eq:phi} gives
\begin{equation}\label{eq:phitheta}
\phi_{\theta Z}(s,\hat t\,)
\;=\;c_s^{\,Z}\!\left(\frac{\theta\hat t}{1-\theta h_Z\hat t}\right)
\;=\;\phi_Z\bigl(s,\theta(s)\,\hat t\,\bigr).
\end{equation}

In terms of permitted extensions this reads as follows. If $u$ is a holomorphic
function on a neighborhood of $C\setminus T_Z$, if $u\big|_C=\theta$ and  if $\bigl(\tilde Z u\bigr)\big|_C=0$, then $u\,\tilde Z$ is a
permitted extension of $\theta Z$.

Indeed $(u\tilde Z)|_C=\theta Z$, and
\begin{align*}
\bigl(\nabla_{u\tilde Z}u\tilde Z\bigr)\big|_C
&=u\bigl(\tilde Zu\bigr)\tilde Z\big|_C
+u^2\bigl(\nabla_{\tilde Z}\tilde Z\bigr)\big|_C
=\theta\bigl(\tilde Zu\bigr)\big|_C\,Z+2\,\theta^2h_Z\,Z\\
&=2h_{\theta Z}(\theta Z).
\end{align*}
\end{remark}

\section{The index}\label{sec:index}

\subsection{The adapted frame and the definition}
Let $C$ be a smooth holomorphic curve (a local curve suffices), $p\in C$,
and $\gamma$ a parametrization of a disc $D\subset C$ with $\gamma(0)=p$.
Let $Z$ be a holomorphic vector field along $D$, not identically tangent to
$C$, whose zeros and tangencies with $C$ inside $D$ are contained in
$\{p\}$. Let $\FF$ be a singular holomorphic foliation defined in a
neighborhood of $D\setminus\{p\}$ in $M$, generically transversal to $Z$
along $C$, whose tangencies with $Z$ along $D$ are contained in $\{p\}$
(after shrinking $D$).

Fix a permitted extension $\tilde Z$ of $Z|_{D\setminus p}$, let
$(s,t)$ be its adapted coordinates and let $\phi$ be the corresponding
biholomorphism. Set $\tilde\gamma:=\phi_*\partial_s$, so that
$(\tilde\gamma,\tilde Z)$ is a commuting holomorphic frame on a
neighborhood of $D\setminus\{p\}$ with $\tilde\gamma|_C=\gamma'$ and
$\tilde Z|_C=Z$. If $X$ is a local generator of $\FF$, write
\[
X\;=\;\lambda\,\tilde\gamma+\mu\,\tilde Z,
\qquad
m\;:=\;\frac{\mu}{\lambda}\,,
\]
the \emph{slope} of $\FF$ in the adapted frame; $m$ is holomorphic near
$(D\setminus p)\times\{0\}$ and does not depend on the generator $X$.

\begin{definition}\label{def:index}
The \emph{index density} of $\FF$ with respect to $Z$ along $C$ is the
holomorphic $1$-form
\[
\eta_Z\;:=\;\partial_t m\big|_{t=0}\;ds,
\]
and the \emph{index} of $\FF$ with respect to $Z$ along $C$ at $p$ is
\[
\Ind(\FF,Z,C,p)\;:=\;\Res_{p}\,\eta_Z
\;=\;\frac{1}{2\pi i}\oint_{|s|=\varepsilon}\partial_t m\big|_{t=0}\,ds.
\]
\end{definition}

The definition involves the choice of a permitted extension $\tilde Z$;
Proposition~\ref{prop:welldef} below shows that $\eta_Z$ does not depend on
it. The residue is well defined because the integrand is holomorphic on the
punctured disc $0<|s|<\varepsilon$; \emph{no assumption whatsoever is made at
$p$ itself} --- neither on $Z$ (which may vanish or be tangent to $C$ at
$p$), nor on $\FF$ (which need not extend across $p$). Note that $\eta_Z$
is a $1$-form on the curve $D\setminus\{p\}$: since $\tilde Z=\partial_t$ in
the adapted coordinates, one has $\eta_Z=\tilde Z(m)\big|_{C}\,ds$ with
$s=\gamma^{-1}$ the coordinate on the curve.

\begin{remark}\label{rem:whennonzero}
The index can be nonzero only when $p$ is a point of one of the following
three kinds: a degeneracy point of the reference ($p\in T_Z$, that is, $Z$
vanishes or is tangent to $C$ at $p$); a point across which $\FF$ is not
assumed to extend; or a tangency point of $\FF$ with the reference, that
is, a zero of $Z\wedge X$ along $C$ --- these include the singular points
of $\FF$ on $C$. Indeed, if $p$ is of none of these kinds, then $Z(p)$ and
$X(p)$ are linearly independent, so $\lambda$ does not vanish near
$(p,0)$ and the slope $m$ is holomorphic there.
\end{remark}

\subsection{Well-definedness}

\begin{proposition}\label{prop:welldef}
The $1$-form $\eta_Z$ does not depend on: (a) the generator $X$ of $\FF$;
(b) the permitted extension $\tilde Z$ within its class; (c) the
parametrization $\gamma$ of $C$; (d) the local representative $\nabla$ of
the projective connection. It depends only on $\FF$, $Z$, $C$ and the
projective connection.
\end{proposition}

\begin{proof}
(a) Under $X\mapsto uX$ both $\lambda,\mu$ are multiplied by $u$ and
$m$ is unchanged.

(b) This can be proved directly, by comparing the adapted coordinates of two
permitted extensions, but it is an immediate consequence of the expression
for $\eta_Z$ obtained in Proposition~\ref{prop:affine} right below: the
right hand side of \eqref{eq:affine} involves only $Z$, $\gamma$, $X$ and
the representative $\nabla$, so that every permitted extension yields the
same $1$-form. 

(c) Under a reparametrization $s=\varphi(u)$ the function $h$, and hence the
permitted class and the coordinate $t$, are unchanged; only
$\tilde\gamma\mapsto\varphi'\,\tilde\gamma$, so $m\mapsto\varphi'm$ and
\[
\partial_t(\varphi'm)\big|_{t=0}\,du
=\varphi'\,\partial_t m\big|_{t=0}\,du
=\partial_t m\big|_{t=0}\,ds .
\]

(d) By Lemma~\ref{lem:extension}(i) the set of permitted extensions of $Z$
relative to $\nabla$ and relative to any other local representative is the
same set, and by (b) the form $\eta_Z$ does not depend on which of its
members is used. Hence $\eta_Z$ is the same for all representatives.
\end{proof}

\subsection{The general expression}\label{subsec:affine}
We now compute the index density in terms of a representative $\nabla$. The
computation makes no use of the particular permitted extension, which is
what proves part (b) of Proposition~\ref{prop:welldef}. Let $\gamma$
parametrize $C$, let $Z$ be a holomorphic field along $\gamma$ as above ---
which may vanish or be tangent to $C$ at $p$ --- and let $X$ be any local
generator of $\FF$. Three tangency functions along $\gamma$ govern the
computation:
\[
\tau\;:=\;Z\wedge X,
\qquad
\kappa\;:=\;\gamma'\wedge X,
\qquad
w\;:=\;\gamma'\wedge Z ,
\]
whose zeros are, respectively, the tangencies of $\FF$ with the reference
(these include $\Sing\FF\cap C$), the tangencies of $\FF$ with $C$, and the
degeneracy set $T_Z$. Recall from \eqref{eq:h} that
$h=(Z\wedge\nabla_{\gamma'}Z)/(Z\wedge\gamma')$.

\begin{proposition}\label{prop:affine}
Let $\tilde Z$ be a permitted extension and let $\eta_Z$ be the index
density computed from it. Then, on the punctured disc,
\begin{equation}\label{eq:affine}
\eta_Z\;=\;\left[\;
\frac{2h\,\kappa\;-\;\bigl(\nabla_{\gamma'}Z\wedge X\bigr)}{\tau}
\;-\;\frac{w}{\tau}\cdot\frac{\nabla_ZX\wedge X}{\tau}
\;\right]ds ,
\end{equation}
all terms being evaluated along $\gamma$. We emphasize that, although a local representative $\nabla$
appears on the right hand side, $\eta_Z$ depends only on the
projective connection.
\end{proposition}

\begin{proof}
The slope of $\FF$ in the adapted frame is determined by
$(\tilde\gamma+m\,\tilde Z)\wedge X=0$, that is, $m=N/D$ with
$N=-(\tilde\gamma\wedge X)$ and $D=\tilde Z\wedge X$. Now $N$ and $D$ are
sections of the same line bundle $\Lambda^2TM$, so that $m=N/D$ is a
function; applying to $N=m\,D$ the covariant derivative
$\nabla_{\tilde Z}$ induced on $\Lambda^2TM$ gives
$\nabla_{\tilde Z}N=\tilde Z(m)\,D+m\,\nabla_{\tilde Z}D$, that is,
\[
\tilde Z(m)=\frac{\nabla_{\tilde Z}N}{D}-m\,\frac{\nabla_{\tilde Z}D}{D}\,.
\]

Along $C$ we have $\tilde\gamma|_C=\gamma'$ and $\tilde Z|_C=Z$, so
$N|_C=-\kappa$ and $D|_C=\tau$. Moreover, 
\[ \bigl(\nabla_{\tilde Z}\tilde Z\bigr)\big|_{C}=2hZ;\qquad 
\bigl(\nabla_{\tilde Z}\tilde\gamma\bigr)\big|_{C}
=\bigl(\nabla_{\tilde\gamma}\tilde Z\bigr)\big|_{C}
=\nabla_{\gamma'}Z.
\]
The first identity is the defining equation of a permitted extension, and
the second one holds because $\nabla$ is torsion-free and
$[\tilde\gamma,\tilde Z]=0$. Hence
\[
\bigl(\nabla_{\tilde Z}N\bigr)\big|_C
=-\bigl(\nabla_{\gamma'}Z\wedge X\bigr)-\bigl(\gamma'\wedge\nabla_ZX\bigr),
\qquad
\bigl(\nabla_{\tilde Z}D\bigr)\big|_C
=2h\tau+\bigl(Z\wedge\nabla_ZX\bigr),
\]
and therefore
\[
\tilde Z(m)\big|_C
=\frac{-\bigl(\nabla_{\gamma'}Z\wedge X\bigr)
-\bigl(\gamma'\wedge\nabla_ZX\bigr)}{\tau}
+\frac{\kappa}{\tau}\Bigl(2h+\frac{Z\wedge\nabla_ZX}{\tau}\Bigr) .
\]
Off the zeros of $\tau$ the pair $(Z,X)$ is a frame; writing
$\nabla_ZX=aZ+bX$ we get $\gamma'\wedge\nabla_ZX=a\,w+b\,\kappa$ and
$Z\wedge\nabla_ZX=b\,\tau$, whence
\[
-\frac{\gamma'\wedge\nabla_ZX}{\tau}
+\frac{\kappa}{\tau}\cdot\frac{Z\wedge\nabla_ZX}{\tau}
=-\,a\,\frac{w}{\tau},
\qquad
a=\frac{\nabla_ZX\wedge X}{\tau},
\]
and, since $\tilde Z=\partial_t$ in the adapted coordinates,
$\tilde Z(m)\big|_C=\partial_t m\big|_{t=0}$; hence \eqref{eq:affine}
follows.
\end{proof}

\subsection{The rescaling law}

\begin{proposition}\label{prop:rescaling}
Let $\theta$ be a holomorphic nonvanishing function on $D\setminus\{p\}$,
meromorphic at $p$. Then
\[
\eta_{\theta Z}\;=\;\eta_Z\;-\;d\log\theta
\qquad\text{on }D\setminus\{p\},
\]
and consequently
$\Ind(\FF,\theta Z,C,p)=\Ind(\FF,Z,C,p)-\ord_p\theta$. In particular the
index is invariant under multiplication of $Z$ by a unit.
\end{proposition}

\begin{proof}
By \eqref{eq:phitheta}, the adapted coordinates of $\theta Z$ are
$(s,\hat t\,)$ with $t=\theta(s)\hat t$. Along a leaf of $\FF$,
$\frac{dt}{ds}=\theta'\hat t+\theta\frac{d\hat t}{ds}$, so the two slopes
are related by $m_Z=\theta'\hat t+\theta\,m_{\theta Z}$, i.e.
$m_{\theta Z}=(m_Z-\theta'\hat t\,)/\theta$. Differentiating at
$\hat t=0$ (where $t=0$),
\[
\partial_{\hat t}m_{\theta Z}\big|_{0}
=\frac{\theta\,\partial_t m_Z|_0-\theta'}{\theta}
=\partial_t m_Z\big|_0-\frac{\theta'}{\theta}\,. \qedhere
\]
\end{proof}

\subsection{The index with respect to a morphism}\label{subsec:psi}
Let $L$ be a holomorphic line bundle on $C$ and
$\psi\colon L\to TM|_C$ a morphism of vector bundles, not identically zero,
whose image is not everywhere tangent to $C$; equivalently, writing
$\pi\colon TM|_C\to N_C$ for the natural projection, the composition
$\pi\circ\psi\colon L\to N_C$ is not identically zero. For $p\in C$ choose a local frame $e$ of $L$ near
$p$ and set $Z:=\psi(e)$, a holomorphic vector field along $C$ whose zeros
and tangencies with $C$ near $p$ are isolated. If $\FF$ is as above, we set
\[
\Ind(\FF,\psi,C,p)\;:=\;\Ind(\FF,Z,C,p);
\]
by Proposition~\ref{prop:rescaling} this does not depend on the frame $e$,
since two frames differ by a unit. When $\psi$ is the inclusion of a line
subbundle $V\subset TM|_C$ we write $\Ind(\FF,V,C,p)$. When
$\psi=(T\GG\to TM)|_C$ for a singular holomorphic foliation $\GG$ with $C$
not $\GG$-invariant, we write $\Ind(\FF,\GG,C,p)$; here $L=T\GG|_C$.

\section{The main theorem}\label{sec:main}

The following is the main result of the paper. It contains Theorem~A of the
introduction, which is the particular case in which $\psi$ is the morphism
$(T\GG\to TM)|_C$ induced by a singular holomorphic foliation $\GG$ with $C$
not $\GG$-invariant: then $L=T\GG|_C$ and $\deg L=T\GG\cdot C$.

\begin{theorem}\label{thm:main}
Let $M$ be a complex surface endowed with a projective connection and let
$C\subset M$ be a compact smooth holomorphic curve. Let
$\psi\colon L\to TM|_C$ be a morphism from a line bundle $L$ on $C$ with
$\pi\circ\psi\not\equiv0$, and let $\FF$ be a singular holomorphic foliation
defined in a neighborhood of $C\setminus\Sigma$, $\Sigma\subset C$ finite,
such that the tangency locus of $\FF$ and $\psi$ along $C$ is a finite set.
Then
$\Ind(\FF,\psi,C,p)$ is defined for every $p\in C$, vanishes outside a
finite set, and
\[
\sum_{p\in C}\Ind(\FF,\psi,C,p)\;=\deg L.
\]

\end{theorem}

Here the tangency locus of $\FF$ and $\psi$ along $C$ is the zero locus on
$C$ of the functions $(X_\FF\wedge Z_i)|_C$, where $X_\FF$ runs over local
generators of $\FF$ and $Z_i=\psi(e_i)$ over local frames; it contains
$\Sing\FF\cap C$ and the zeros of $\psi$.

We denote by
\begin{equation}\label{eq:S}
S\;:=\;T_\psi\;\cup\;\Sigma\;\cup\;\bigl(\TTang(\FF,\psi)\cap C\bigr)
\end{equation}
the union of the three kinds of points listed in
Remark~\ref{rem:whennonzero}, where $T_\psi$ denotes the zero locus of
$\pi\circ\psi$, that is, the set of tangencies and zeros of $\psi$; it is
finite, being the zero locus of a section of $\mathrm{Hom}(L,N_C)$ which is
not identically zero. The tangency locus of
$\FF$ and $\psi$ is finite by hypothesis, and $\Sigma$ is finite; hence $S$
is finite.

\begin{proof}[Proof of Theorem~\ref{thm:main}]
By Remark~\ref{rem:whennonzero}, $\Ind(\FF,\psi,C,p)$ is defined for every
$p\in C$ and vanishes outside the finite set $S$.
Cover $C$ by finitely many discs $V_i$ with frames $e_i$ of $L$, and let
$\theta_{ij}$ be the transitions: $e_i=\theta_{ij}e_j$, hence
$Z_i=\theta_{ij}Z_j$ with $\theta_{ij}$ holomorphic and nonvanishing. By
Proposition~\ref{prop:rescaling},
\begin{equation}\label{eq:cocycle}
\eta_i\;=\;\eta_j\;-\;d\log\theta_{ij}
\qquad\text{on }(V_i\cap V_j)\setminus S .
\end{equation}
Choose a global meromorphic section $\zeta$ of $L$, $\zeta=f_ie_i$
with $f_i$ meromorphic; from $e_i=\theta_{ij}e_j$ one gets
$f_i/f_j=\theta_{ij}^{-1}$, hence
$d\log f_i-d\log f_j=-\,d\log\theta_{ij}$. By \eqref{eq:cocycle} the local
$1$-forms
\[
\chi\big|_{V_i}\;:=\;\eta_i\;-\;d\log f_i
\]
agree on overlaps and define a global $1$-form $\chi$ on $C$, holomorphic
outside the finite set $S\cup|\operatorname{div}\zeta|$. At every $p\in C$,
\[
\Res_p\chi\;=\;\Ind(\FF,\psi,C,p)\;-\;\ord_p\zeta.
\]
 Therefore, using  the residue theorem,
\[
\sum_p\Ind(\FF,\psi,C,p)=\sum_p\Res_p\chi + \sum_p \ord_p\zeta=0+ \deg L . \qedhere
\]
\end{proof}

\subsection{The invariant case}
Recall the index of \cite{CS} for an invariant curve: if $C=\{y=0\}$ in local coordinates
is invariant by $\FF$, and the leaves of $\FF$ satisfy $dy/dx=m(x,y)$ --- with
$m(x,0)\equiv0$, then
\[
\CS(\FF,C,p)\;:=\;\Res_{p}\;\partial_ym(x,0)\,dx .
\]
It is well known that this index does not depend on the choice
of such coordinates.

\begin{proposition}\label{prop:invariant}
Assume $C$ is $\FF$-invariant. Then for every $p\in C$,
\[
\Ind(\FF,\psi,C,p)\;=\;\CS(\FF,C,p)\;-\;\ord_p(\pi\circ\psi) .
\]In particular the index does not depend on the projective connection.
\end{proposition}

\begin{proof}
Let $\gamma$ parametrize $C$ near $p$, let $X$ be a local generator of
$\FF$, let $e$ be a local frame of $L$ around $p$ and put $Z=\psi(e)$, so
that $\Ind(\FF,\psi,C,p)=\Ind(\FF,Z,C,p)$.

Since $C$ is $\FF$-invariant, $X$ is tangent to $C$ along $C$, that is,
\[
X\big|_C=u\,\gamma'
\]
for some holomorphic $u$, nonvanishing on the punctured disc: there
$\tau=Z\wedge X=-u\,(\gamma'\wedge Z)=-u\,w$, and both $\tau$ and $w$ are
nonvanishing. Moreover, $\kappa=\gamma'\wedge X$ vanishes identically
along $C$, so that the term $2h\kappa/\tau$ of \eqref{eq:affine} drops out
and Proposition~\ref{prop:affine} reads
\begin{equation}\label{eq:invexpr}
\eta_Z\;=\;\left[\,
-\,\frac{\nabla_{\gamma'}Z\wedge X}{\tau}
\;-\;\frac{w}{\tau}\cdot\frac{\nabla_ZX\wedge X}{\tau}
\,\right]ds ,
\end{equation}
all terms being evaluated along $\gamma$.

Let $Z_0$ be a holomorphic vector field along $C$ near $p$ such that
$\bigl(Z_0,\gamma'\bigr)$ is a frame along $C$; in particular $Z_0$ is
transversal to $C$. Then the
reference decomposes as
\[
Z\;=\;\theta_1\,Z_0\;+\;\theta_2\,\gamma',
\]
where $\theta_1$ and $\theta_2$ are holomorphic and
$$\ord_p\theta_1=\ord_p(\gamma'\wedge Z)=\ord_p(\pi\circ\psi).$$ We use two invariance
properties of \eqref{eq:invexpr}.

\emph{(a) Tangential invariance:} for $\theta$ holomorphic,
$\eta_{Z+\theta\gamma'}=\eta_Z$. Indeed the substitution
$Z\mapsto Z+\theta\gamma'$ leaves $w$ unchanged, and leaves $\tau$
unchanged because the extra term is $\theta\kappa=0$; while
\begin{align*}
\nabla_{\gamma'}(Z+\theta\gamma')\wedge X
&=\nabla_{\gamma'}Z\wedge X
+\theta'\,\underbrace{(\gamma'\wedge X)}_{=\,0}
+\theta\,(\nabla_{\gamma'}\gamma'\wedge X),\\
\nabla_{Z+\theta\gamma'}X\wedge X
&=\nabla_ZX\wedge X+\theta\,(\nabla_{\gamma'}X\wedge X).
\end{align*}
Setting  $\Theta:=\nabla_{\gamma'}\gamma'\wedge\gamma'$, from $X|_C=u\gamma'$ one gets
$\nabla_{\gamma'}\gamma'\wedge X=u\,\Theta$ and
$\nabla_{\gamma'}X\wedge X=(u'\gamma'+u\nabla_{\gamma'}\gamma')\wedge
u\gamma'=u^2\,\Theta$, and the
two extra contributions to \eqref{eq:invexpr},
\[
-\,\frac{\theta\,u\,\Theta}{\tau}
\qquad\text{and}\qquad
-\,\frac{w}{\tau}\cdot\frac{\theta\,u^{2}\Theta}{\tau}
=+\,\frac{\theta\,u\,\Theta}{\tau}
\qquad\bigl(\tau=-uw\bigr),
\]
cancel.

\emph{(b) Rescaling:} $\eta_{\theta_1Z_0}=\eta_{Z_0}-d\log\theta_1$ on the
punctured disc, by Proposition~\ref{prop:rescaling}.

Combining,
\[
\eta_Z
\overset{(a)}{=}\eta_{\theta_1Z_0}
\overset{(b)}{=}\eta_{Z_0}-d\log\theta_1 ,
\]
and taking residues at $p$,
\[
\Ind(\FF,\psi,C,p)=\Ind(\FF,Z,C,p)
=\Ind(\FF,Z_0,C,p)-\ord_p(\pi\circ\psi) .
\]

It remains to identify $\Ind(\FF,Z_0,C,p)$ with the Camacho--Sad index.
Since $Z_0$ is transversal to $C$, its degeneracy set is empty near $p$, so
by Lemma~\ref{lem:extension}(ii) the map $\phi$ of \eqref{eq:phi} is a
biholomorphism onto a neighborhood of $p$ in $M$; that is, the adapted
coordinates $(s,t)$ are holomorphic coordinates near $p$ in which
$C=\{t=0\}$ and $\gamma(s)=(s,0)$. In them the leaves of $\FF$ satisfy
$dt/ds=m$, with $m(s,0)\equiv0$ because $C$ is invariant, so that $m$ is
the slope field of $\FF$ in the sense recalled above and
$\eta_{Z_0}=\partial_tm\big|_{t=0}\,ds$ is the Camacho--Sad density
computed in the coordinates $(s,t)$. Therefore
\[
\Ind(\FF,Z_0,C,p)\;=\;\CS(\FF,C,p),
\]
and the proposition follows.
\end{proof}

\begin{corollary}[\cite{CS}]\label{cor:CS}
If $C$ is compact, smooth and $\FF$-invariant, then
$\sum_{p}\CS(\FF,C,p)=C\cdot C$.
\end{corollary}

\begin{proof}
We can find a line bundle $L$ over $C$ with a  morphism $\psi\colon L\to TM|_C$ such that $\pi\circ\psi\colon L\to  N_C$ is not 
identically zero. Then 
$\sum_p\ord_p(\pi\circ\psi)=\deg N_C-\deg L=C\cdot C-\deg L$. Therefore, by
Proposition~\ref{prop:invariant} and Theorem~\ref{thm:main},
\begin{align*}
\sum_p\CS(\FF,C,p)
&=\sum_p\Ind(\FF,\psi,C,p)+\sum_p\ord_p(\pi\circ\psi)\\
&=\deg L+\bigl(C\cdot C-\deg L\bigr)\;=\;C\cdot C. \qedhere
\end{align*}
\end{proof}

\section{Local computations}\label{sec:local}

Throughout this section the standing hypotheses of
Definition~\ref{def:index} are in force, and $\nabla$ denotes a local
representative of the projective connection near the point under
consideration. We
compute the index at the various kinds of points of the set $S$ of
\eqref{eq:S}, under natural nondegeneracy assumptions.

\subsection{The contact invariant}

Let $V\subset TM|_C$ be a line subbundle tangent to $C$ at $p$, that is,
with $V_p=T_pC$. For $v\in T_pC$ choose a section $X$ of $V$ near $p$ with
$X(p)=v$ and set
\[
q_V(v)\;:=\;\nabla_vX\ \ \mathrm{mod}\ T_pC\;\in\;N_p ,
\]
the derivative being that of a section of $TM|_C$ along $C$.

This does not depend on the choice of $X$: if $v\neq0$ then $X$ generates
$V$ near $p$, any other such section is $uX$ with $u(p)=1$, and
$\nabla_v(uX)=v(u)\,X(p)+\nabla_vX$, whose extra term $v(u)\,v$ lies in
$T_pC$; if $v=0$ both derivatives vanish. Nor does it depend on the local
representative: if $\nabla'$ is projectively equivalent to $\nabla$ then
$\nabla'_vX=\nabla_vX+\xi(v)X(p)+\xi(X(p))\,v=\nabla_vX+2\,\xi(v)\,v$, again
a change lying in $T_pC$. 

Finally $cX$ is a section of $V$ with value $cv$ at $p$ and
$\nabla_{cv}(cX)=c^2\,\nabla_vX$, so that $q_V(cv)=c^2q_V(v)$. A map between
one-dimensional spaces which is homogeneous of degree two is the restriction
to the diagonal of a unique symmetric bilinear map; there is therefore a
unique element $\sigma_V(p)\in N_p\otimes(T_p^*C)^{\otimes2}$ such that
\[
\sigma_V(p)(v,v)\;=\;q_V(v)
\qquad\text{for all } v\in T_pC .
\]
Thus $\sigma_V(p)$ is attached to $V$, $C$ and $p$ alone; no parametrization
intervenes. For $V=TC$ it is the value at $p$ of the projective second
fundamental form of Lemma~\ref{lem:sff}: $\sigma_{TC}(p)=\sigma_C(p)$.

A parametrization $\gamma$ of $C$ near $p$ determines an isomorphism of
one-dimensional vector spaces
\begin{equation}\label{eq:Phi}
\Phi\colon N_p\otimes(T_p^*C)^{\otimes2}\longrightarrow\Lambda^2T_pM,
\qquad
\Phi(\sigma)\;=\;\gamma'(p)\wedge\sigma\bigl(\gamma'(p),\gamma'(p)\bigr).
\end{equation}
We keep it explicit wherever it is applied. Since $\Phi$ is a linear
isomorphism between one-dimensional spaces,
\begin{equation}\label{eq:Phiratio}
\frac{\Phi(a)}{\Phi(b)}\;=\;\frac{a}{b}
\qquad\text{for } b\neq0,
\end{equation}
so that a ratio may be read indifferently on either side. 
\begin{definition}\label{def:k}
Let $V\subset TM|_C$ be a line subbundle tangent to $C$ at $p$, with
isolated tangency. The \emph{contact invariant} of $V$ and $C$ at $p$ is
\[
k(V,C,p)\;:=\;\frac{\sigma_V(p)}{\sigma_C(p)}\;\in\;\PP^1,
\]
defined whenever $\bigl(\sigma_V(p),\sigma_C(p)\bigr)\neq(0,0)$. The
quotient is a number because $\sigma_V(p)$ and $\sigma_C(p)$ lie in the same
one-dimensional space $N_p\otimes(T_p^*C)^{\otimes2}$.
\end{definition}

Let $Z$ be a generator of $V$ near $p$, normalized so that
$Z(p)=\gamma'(p)$. The tangency function $w=\gamma'\wedge Z$ is a
holomorphic section of the line bundle $\Lambda^{2}TM$ along $C$, and its
derivative $w':=\nabla_{\gamma'}w$ is taken with the connection induced on
$\Lambda^{2}TM$ by a local representative $\nabla$. Any two connections on
a line bundle differ by a $1$-form, say
$\nabla'_{v}s=\nabla_{v}s+\beta(v)\,s$; at a zero of $s$ the two
derivatives therefore agree. Since $Z(p)=\gamma'(p)$ gives $w(p)=0$, the
value
\[
w'(p)\;\in\;\Lambda^{2}T_pM
\]
is well defined, independently of the representative of the projective
connection.

\begin{lemma}\label{lem:simple}
With $\Phi$ as in \eqref{eq:Phi},
\[
w'(p)\;=\;\Phi\bigl(\sigma_V(p)-\sigma_C(p)\bigr).
\]
In particular $w$ has a simple zero at $p$ --- we then say that the tangency
of $V$ with $C$ at $p$ is \emph{simple} --- if and only if $k(V,C,p)\neq1$.
\end{lemma}

\begin{proof}
$w'=\nabla_{\gamma'}\gamma'\wedge Z+\gamma'\wedge\nabla_{\gamma'}Z$. At $p$,
where $Z(p)=\gamma'(p)$, the first summand is
\[
\bigl(\nabla_{\gamma'}\gamma'\bigr)(p)\wedge\gamma'(p)
=-\,\gamma'(p)\wedge\bigl(\nabla_{\gamma'}\gamma'\bigr)(p)
=-\,\Phi\bigl(\sigma_C(p)\bigr),
\]
and the second is
\[
\gamma'(p)\wedge\bigl(\nabla_{\gamma'}Z\bigr)(p)=\Phi\bigl(\sigma_V(p)\bigr).
\]
Adding the two equalities and using the
linearity of $\Phi$ gives the formula. Finally, $\Phi$ is an isomorphism,
so $w'(p)\neq0$ if and only if $\sigma_V(p)\neq\sigma_C(p)$, that is, if and
only if $k(V,C,p)\neq1$.
\end{proof}

\begin{remark}\label{rem:kvalues}
(a) When $V=T\GG|_C$ comes from a foliation $\GG$ tangent to $C$ at $p$ we
write $k(\GG,C,p):=k(T\GG|_C,C,p)$; since the leaf $\mathcal L$ of $\GG$
through $p$ is tangent to $V$ at $p$ to first order,
$\sigma_V(p)=\sigma_{\mathcal L}(p)$ and
\[
k(\GG,C,p)\;=\;\frac{\sigma_{\mathcal L}(p)}{\sigma_C(p)}\,.
\]

(b) This is a genuine \emph{ratio of curvatures}. The invariance arguments
of Lemma~\ref{lem:sff} and of the definition of $\sigma_V$ use only that the
extra term $\xi(X)Y+\xi(Y)X$ lies in $TC$, so they apply verbatim when
$\xi$ is merely a smooth $(1,0)$-form; $\sigma$ and $k$ may therefore be
computed with a smooth connection differing from a representative by such a
term. On $\PP^2$ the Chern connection of the Fubini--Study metric, whose
Christoffel symbols are
$\Gamma^k_{ij}=-(\delta^k_i\bar z_j+\delta^k_j\bar z_i)/(1+|z|^2)$, is of
this form relative to the flat representative of an affine chart, and
$k(\GG,C,p)$ becomes the quotient of the curvature vectors of the leaf and
of the curve at $p$, measured in the Fubini--Study metric. The individual
curvatures are not projective invariants; their ratio is.

(c) $k=0$ if and only if $\sigma_V(p)=0$, that is, if and only if $V$ has
contact of order $\ge3$ with its tangent geodesic at $p$; $k=\infty$ if and
only if $\sigma_C(p)=0$, i.e.\ if and only if $p$ is an inflection point of
$C$. We use the convention $\tfrac1{1-\infty}=0$.

(d) In coordinates in which the geodesic tangent to $C$ at $p$ is the
$x$-axis --- in the flat case, in affine coordinates of a distinguished
chart with $T_pC=\{y=0\}$ --- write $C$ as the graph of $y=f(x)$ and
$V=\CC\cdot(1,\rho(x))$ along $C$, with $f(0)=f'(0)=\rho(0)=0$. Then
$\sigma_C(p)$ and $\sigma_V(p)$ are $f''(0)$ and $\rho'(0)$ times
$dx^{\otimes2}\otimes\partial_y$, so
\[
k(V,C,p)=\frac{\rho'(0)}{f''(0)}\,.
\]
\end{remark}

\subsection{Tangencies of the reference with the curve}

\begin{proposition}\label{prop:case1}
Let $\psi$ be nonvanishing at $p$, with $V:=\operatorname{im}\psi$ tangent
to $C$ at $p$ with a simple tangency, and let $\FF$ be holomorphic at $p$
and transversal to $C$ at $p$. Then
\[
\Ind(\FF,\psi,C,p)\;=\;\frac{2\,k}{1-k}\,,\qquad k=k(V,C,p).
\]
In particular the index does not depend on $\FF$.
\end{proposition}

\begin{proof}
Since $Z=\psi(e)$ is nonzero and tangent to $C$ at $p$, after multiplying
the frame $e$ by a constant we may assume $Z(p)=\gamma'(p)$; by
Proposition~\ref{prop:rescaling} this does not change the index. Let $X$ be
a generator of $\FF$ near $p$. As $\FF$ is transversal to $C$ at $p$ we have
$\kappa(p)=(\gamma'\wedge X)(p)\neq0$, and $\tau(p)=(Z\wedge X)(p)=\kappa(p)$
is nonzero as well.

In \eqref{eq:affine} the term $(\nabla_{\gamma'}Z\wedge X)/\tau$ and the
term $(w/\tau)\,(\nabla_ZX\wedge X)/\tau$ are therefore holomorphic at $p$. Hence
\[
\Ind(\FF,\psi,C,p)\;=\;\Res_p\frac{2h\,\kappa}{\tau}\,ds,
\qquad
h=-\,\frac{Z\wedge\nabla_{\gamma'}Z}{w}\,.
\]
At $p$ we have $(\kappa/\tau)(p)=1$ and, since $Z(p)=\gamma'(p)$,
\[
\bigl(Z\wedge\nabla_{\gamma'}Z\bigr)(p)
=\bigl(\gamma'\wedge\nabla_{\gamma'}Z\bigr)(p)=\Phi\bigl(\sigma_V(p)\bigr),
\]
while $w'(p)=\Phi\bigl(\sigma_V(p)-\sigma_C(p)\bigr)\neq0$ by
Lemma~\ref{lem:simple}. As $w$ has a simple zero, this and
\eqref{eq:Phiratio} give
\[
\Ind(\FF,\psi,C,p)
=-\,\frac{2\,\Phi\bigl(\sigma_V(p)\bigr)}{w'(p)}
=-\,\frac{2\,\sigma_V(p)}{\sigma_V(p)-\sigma_C(p)}
=\frac{2k}{1-k}\,.
\]
The generator $X$ has disappeared from the computation. The formula also
covers $k=\infty$ ($\sigma_C(p)=0$), giving the value $-2$.
\end{proof}

\subsection{Geodesic foliations of a tangent line field}

A holomorphic foliation is \emph{geodesic} if its leaves are geodesics of
the projective connection. The following lemma attaches such a foliation,
defined near $C$, to a line subbundle $V\subset TM|_C$.

\begin{lemma}\label{lem:geodfol}
Let $V\subset TM|_C$ be a line subbundle with isolated tangencies
$T_V\subset C$. The geodesics tangent to $V$ along $C$ are the leaves of a
holomorphic foliation $\FF_V$, the \emph{geodesic foliation generated by
$V$}, defined on a neighborhood of $C\setminus T_V$; one has
$T\FF_V|_{C\setminus T_V}=V$, and $\FF_V$ is generated by the permitted
extension $\tilde Z_V$ of any local generator $Z_V$ of $V$.
\end{lemma}

\begin{proof}
Let $Z_V$ be a local generator of $V$ and $\phi$ as in \eqref{eq:phi}. By
Lemma~\ref{lem:extension}(ii), $\phi$ is a biholomorphism near
$(C\setminus T_V)\times\{0\}$ whose curves $s=\mathrm{const}$ are the
geodesics tangent to $V$; their images therefore foliate a neighborhood of
each point of $C\setminus T_V$, and $\tilde Z_V=\phi_*\partial_t$ is tangent
to them. The local foliations glue, since the geodesic through $\gamma(s)$
tangent to $V(s)$ does not depend on any choice.
\end{proof}

\begin{proposition}\label{prop:case2}
Let $V\subset TM|_C$ be tangent to $C$ at $p$ with a simple tangency, let
$\FF=\FF_V$ be the geodesic foliation generated by $V$ on a neighborhood of
$C\setminus\{p\}$, and let $\psi$ be transversal to $C$ at $p$. Then
\[
\Ind(\FF_V,\psi,C,p)\;=\;\frac{k}{1-k}\,,
\qquad k=k(V,C,p).
\]
In particular the index does not depend on $\psi$.
\end{proposition}

\begin{proof}
Let $Z_V$ be a generator of $V$ near $p$, normalized by
$Z_V(p)=\gamma'(p)$, and let $h_V$ be the function \eqref{eq:h} attached to
it. Take $X:=\tilde Z_V$, which generates $\FF_V$ by
Lemma~\ref{lem:geodfol}. Along $C$ we have $X|_C=Z_V$ and, by the defining
equation of a permitted extension, $(\nabla_XX)|_C=2h_VZ_V$. Write $Z=\lambda\gamma'+\mu Z_V$ along
$C$, so that
\[
\kappa=\gamma'\wedge Z_V,\qquad \tau=\lambda\kappa,\qquad w=\mu\kappa .
\]
Then
\[
\nabla_ZX\wedge X\big|_C=\bigl(\lambda\,\nabla_{\gamma'}Z_V+2\mu h_V\,Z_V\bigr)\wedge X\big|_C=\lambda\nabla_{\gamma'}Z_V\wedge Z_V
=\lambda h_V\kappa=\tau\,h_V,
\]
the third equality coming
from $h_V=(Z_V\wedge\nabla_{\gamma'}Z_V)/(Z_V\wedge\gamma')$. Hence the
second term of \eqref{eq:affine} equals $-(w/\tau)\,h_V$.

Since $\psi$ is transversal to $C$ at $p$, $w$ is a unit and $h$ is
holomorphic at $p$; moreover $\tau(p)=(Z\wedge Z_V)(p)=-w(p)\neq0$, so the
first term of \eqref{eq:affine} is holomorphic. Therefore
\[
\Ind(\FF_V,\psi,C,p)=\Res_p\Bigl[-\frac{w}{\tau}\,h_V\Bigr]ds
=-\frac{w(p)}{\tau(p)}\,\Res_p\bigl(h_V\,ds\bigr)=\Res_p\bigl(h_V\,ds\bigr),
\]
$h_V$ having a simple pole at $p$. Finally $\kappa=\gamma'\wedge Z_V$ has a
simple zero at $p$ with $\kappa'(p)=\Phi(\sigma_V(p)-\sigma_C(p))$
(Lemma~\ref{lem:simple}), and
$(Z_V\wedge\nabla_{\gamma'}Z_V)(p)=\Phi(\sigma_V(p))$, so
\[
\Res_p\bigl(h_V\,ds\bigr)
=-\frac{\Phi\bigl(\sigma_V(p)\bigr)}{\Phi\bigl(\sigma_V(p)-\sigma_C(p)\bigr)}
\overset{\eqref{eq:Phiratio}}{=}\frac{\sigma_V(p)}{\sigma_C(p)-\sigma_V(p)}
=\frac{k}{1-k}\,. \qedhere
\]
\end{proof}

\subsection{Tangencies of the foliation with the reference}

For line subbundles (or nonvanishing local line fields) $A,B$ along $C$
which are not identically equal near $p$, we write
$\TTang(A,B,p):=\ord_p\bigl(Z_A\wedge Z_B\bigr)$ for local generators
$Z_A,Z_B$; this order does not depend on the generators.

Throughout this subsection $\FF$ is holomorphic at $p$ and transversal to
$C$ at $p$, and $\psi$ is transversal to $C$ at $p$, with $T\FF|_C$ and
$\operatorname{im}\psi$ tangent at $p$. Let $X$ be a generator of $\FF$
near $p$; then $\kappa=\gamma'\wedge X$ and $w=\gamma'\wedge Z$ are units,
so that $(\gamma',X)$ is a frame near $p$, while $\tau=Z\wedge X$ vanishes
at $p$ to the order $\TTang\bigl(T\FF|_C,\psi,p\bigr)$.

One new function is attached to the configuration. It records the failure
of the leaves of $\FF$ to be geodesics: along $C$ we write
\begin{equation}\label{eq:nu}
\nabla_XX\;\equiv\;\nu\,\gamma' \qquad\text{modulo } T\FF ,
\end{equation}
the component of $\nabla_XX$ along $X$ playing no role in what follows. If
$\FF$ is geodesic then $\nabla_XX$ is proportional to $X$, so $\nu\equiv0$.
The function $\nu$ depends on the generator, but the combination occurring
below does not: under $X\mapsto uX$ one has $\nu\mapsto u^{2}\nu$ and
$\tau\mapsto u\tau$, while $w$ is unchanged, so that $(w/\tau)^{2}\nu$ is
independent of $X$.

\begin{proposition}\label{prop:case4}
In the notation above,
\[
\Ind(\FF,\psi,C,p)\;=\;\TTang\bigl(T\FF|_C,\;\psi,\;p\bigr)
\;-\;\Res_p\Bigl(\frac{w}{\tau}\Bigr)^{\!2}\nu\;ds .
\]
\end{proposition}

\begin{proof}
Write the reference in the frame $(\gamma',X)$ as
\[
Z=\lambda\,\gamma'+\mu\,X,
\qquad\text{so that}\qquad
\tau=\lambda\kappa,\quad w=\mu\kappa ,
\]
with $\mu$ a unit and $\ord_p\lambda=\ord_p\tau$, and abbreviate
$G:=\nabla_{\gamma'}\gamma'$, $Y:=\nabla_{\gamma'}X$. We compute the two
terms of \eqref{eq:affine} separately. Differentiating
$Z=\lambda\gamma'+\mu X$ along $\gamma$ gives
\[
\nabla_{\gamma'}Z=\lambda'\gamma'+\lambda\,G+\mu'X+\mu\,Y ,
\]
and wedging this with $X$ and with $Z$, using $\gamma'\wedge X=\kappa$,
\begin{align*}
\nabla_{\gamma'}Z\wedge X&=\lambda'\kappa+\lambda\,(G\wedge X)
+\mu\,(Y\wedge X),\\
Z\wedge\nabla_{\gamma'}Z&=\lambda^{2}(\gamma'\wedge G)
+(\lambda\mu'-\mu\lambda')\,\kappa\\
&\qquad+\lambda\mu\,(\gamma'\wedge Y)+\lambda\mu\,(X\wedge G)
+\mu^{2}(X\wedge Y).
\end{align*}
Dividing the first of these by $\tau=\lambda\kappa$,
\[
-\,\frac{\nabla_{\gamma'}Z\wedge X}{\tau}
=-\frac{\lambda'}{\lambda}-\frac{G\wedge X}{\kappa}
-\frac{\mu\,(Y\wedge X)}{\lambda\kappa}\,;
\]
and since $Z\wedge\gamma'=-w=-\mu\kappa$, so that
$h=-\bigl(Z\wedge\nabla_{\gamma'}Z\bigr)/\mu\kappa$, the second gives
\[
\frac{2h\,\kappa}{\tau}
=-\frac{2\lambda\,(\gamma'\wedge G)}{\mu\kappa}
+\frac{2\lambda'}{\lambda}-\frac{2\mu'}{\mu}
-\frac{2\,(\gamma'\wedge Y)}{\kappa}
-\frac{2\,(X\wedge G)}{\kappa}
-\frac{2\mu\,(X\wedge Y)}{\lambda\kappa}\,.
\]
Adding the two, and using $X\wedge G=-(G\wedge X)$ and
$X\wedge Y=-(Y\wedge X)$, the first term of \eqref{eq:affine} equals
\[
\frac{\lambda'}{\lambda}
-\frac{2\lambda\,(\gamma'\wedge G)}{\mu\kappa}
-\frac{2\mu'}{\mu}
-\frac{2\,(\gamma'\wedge Y)}{\kappa}
+\frac{G\wedge X}{\kappa}
+\frac{\mu\,(Y\wedge X)}{\lambda\kappa}\,.
\]
For the second term, $\nabla_ZX=\lambda\,Y+\mu\,\nabla_XX$, whence, by
\eqref{eq:nu} and $\mu\kappa=w$,
\[
\nabla_ZX\wedge X=\lambda\,(Y\wedge X)+\mu\nu\,\kappa
=\lambda\,(Y\wedge X)+w\,\nu ,
\]
the component of $\nabla_XX$ along $X$ dying against $X$. Since
$\lambda/\tau=1/\kappa$, this gives
\[
-\,\frac{w}{\tau}\cdot\frac{\nabla_ZX\wedge X}{\tau}
=-\,\frac{w}{\tau}\cdot\frac{Y\wedge X}{\kappa}
\;-\;\Bigl(\frac{w}{\tau}\Bigr)^{\!2}\nu
=-\,\frac{\mu\,(Y\wedge X)}{\lambda\kappa}
\;-\;\Bigl(\frac{w}{\tau}\Bigr)^{\!2}\nu\,.
\]
The two terms in $\mu\,(Y\wedge X)$ cancel, leaving
\begin{equation}\label{eq:case3}
\eta_Z=\left[\;\frac{\lambda'}{\lambda}
\;-\;\frac{2\lambda\,(\gamma'\wedge G)}{\mu\kappa}
\;-\;\frac{2\mu'}{\mu}
\;-\;\frac{2\,(\gamma'\wedge Y)}{\kappa}
\;+\;\frac{G\wedge X}{\kappa}
\;-\;\Bigl(\frac{w}{\tau}\Bigr)^{\!2}\nu\;\right]ds .
\end{equation}
Since $\kappa$ and $\mu$ are units and $G$, $Y$ are holomorphic at $p$,
every term but the first and the last is holomorphic there, so
\[
\Ind(\FF,\psi,C,p)=\Res_p\frac{\lambda'}{\lambda}\,ds
-\Res_p\Bigl(\frac{w}{\tau}\Bigr)^{\!2}\nu\;ds
=\TTang\bigl(T\FF|_C,\psi,p\bigr)
-\Res_p\Bigl(\frac{w}{\tau}\Bigr)^{\!2}\nu\;ds . \qedhere
\]
\end{proof}

\begin{corollary}\label{cor:case3}
If $\FF$ is geodesic, then
\[
\Ind(\FF,\psi,C,p)\;=\;\TTang\bigl(T\FF|_C,\;\psi,\;p\bigr).
\]
\end{corollary}

\begin{proof}
The leaves of $\FF$ being geodesics, $\nu\equiv0$.
\end{proof}

\begin{remark}
When $\psi=(T\GG\to TM)|_C$ for a foliation $\GG$ and $C$ is smooth,
$\TTang(T\FF|_C,\psi,p)$ equals the intersection multiplicity
$i_p\bigl(C,\TTang(\FF,\GG)\bigr)$ of $C$ with the tangency curve of the
two foliations.
\end{remark}

\section{A contact formula for foliations and curves}\label{sec:contact1}

In this section $C\subset M$ is a compact smooth holomorphic curve of genus
$g$. Taking determinants in the exact sequence
$0\to TC\to TM|_C\to N_C\to0$ identifies $\Lambda^{2}TM|_C\;\cong\;TC\otimes N_C$, whence
\begin{equation}\label{eq:det}
\deg\Lambda^{2}TM|_C=\deg TC +\deg N_C= (2-2g)+C\cdot C.
\end{equation}

Let $A,B\subset TM|_C$ be line subbundles with $A\wedge B\not\equiv0$. The
exterior product is then a nonzero morphism
$A\otimes B\to\Lambda^{2}TM|_C$, that is, a nonzero section of
$A^{*}\otimes B^{*}\otimes\Lambda^{2}TM|_C$, and we write
$\TTang(A,B):=\deg\operatorname{div}(A\wedge B)$ for the degree of its
vanishing divisor. Then
\begin{equation}\label{eq:tangAB}
\TTang(A,B)\;=\;\deg\Lambda^{2}TM|_C-\deg A-\deg B .
\end{equation}

Taking $B=TC$ gives the \emph{tangency number} of a line subbundle
$V\subset TM|_C$, not everywhere tangent to $C$, with the curve. Indeed,
under \eqref{eq:det} the morphism $V\otimes TC\to\Lambda^{2}TM|_C$
corresponds to $\pi|_V\otimes\mathrm{id}_{TC}$, where
\[
\pi|_V\colon V\hookrightarrow TM|_C\xrightarrow{\ \pi\ }N_C ,
\]
so the two vanish on the same divisor, the points where $V$ falls into
$TC$; we write $\TTang(V,C):=\TTang(V,TC)=\deg\operatorname{div}(\pi|_V)$,
the number of tangencies of $V$ with $C$ counted with multiplicity. By
\eqref{eq:tangAB} and \eqref{eq:det},
\begin{equation}\label{eq:tangVC}
\TTang(V,C)\;=\;\deg N_C-\deg V\;=\;C\cdot C-\deg V .
\end{equation}
We write $T_V\subset C$ for the (finite) set of these tangency points, so
that $\TTang(V,C)=\#T_V$ when all of them are simple.

\begin{lemma}[Tangency lemma]\label{lem:tangencies}
Let $V,W\subset TM|_C$ be line subbundles, neither everywhere tangent to
$C$, with $V\wedge W\not\equiv0$. Then
\[
\TTang(V,W)\;=\;\TTang(V,C)+\TTang(W,C)+2-2g-C\cdot C .
\]
\end{lemma}

\begin{proof}
The three line subbundles $V$, $W$ and $TC$ of $TM|_C$ give three tangency
numbers, all of the shape \eqref{eq:tangAB}. Subtracting those of the pairs
$(V,TC)$ and $(W,TC)$ from that of $(V,W)$, the degrees of $V$ and of $W$
cancel and there remains
\[
\TTang(V,W)-\TTang(V,C)-\TTang(W,C)
\;=\;-\deg\Lambda^{2}TM|_C+2\deg TC ,
\]
which by \eqref{eq:det} equals $\deg TC-\deg N_C=2-2g-C\cdot C$. (For the
analogous computation for a foliation and a curve see
\cite[Prop.~2.2]{Br}.)
\end{proof}

\begin{lemma}[Existence of auxiliary subbundles]\label{lem:existW}
Let $A\subset C$ be finite. Then there exists a line subbundle
$W\subset TM|_C$ all of whose tangencies with $C$ are simple and lie
outside $A$. If moreover a line subbundle $V\subset TM|_C$ is given, $W$ may
be chosen so that in addition $V\wedge W\not\equiv0$.
\end{lemma}

This is a standard genericity argument, and we leave the proof to the
reader.

\begin{theorem}\label{thm:contact1}
Let $M$ be a complex surface endowed with a projective connection, let
$C\subset M$ be a compact smooth holomorphic curve of genus $g$, and let
$V\subset TM|_C$ be a line subbundle, not everywhere tangent to $C$, having
only simple tangencies with $C$, at the points $p_1,\dots,p_n$. Then
\[
\sum_{j=1}^{n}\frac{1}{1-k(V,C,p_j)}\;=\;\frac23\,\bigl(C\cdot C+g-1\bigr).
\]
\end{theorem}

\begin{proof}
Write $T_V=\{p_1,\dots,p_n\}$, $n_V=n=\TTang(V,C)$, and
\[
\Pi_V\;:=\;\sum_{p\in T_V}\frac{1}{1-k(V,C,p)}\,.
\]
Choose $W$ as in Lemma~\ref{lem:existW} with $A=T_V$; write
$T_W$ for its (simple) tangencies with $C$, $n_W=\TTang(W,C)$, $\Pi_W$ for
the corresponding weighted sum, and $T$ for the vanishing locus of
$V\wedge W$. The three sets $T$, $T_V$, $T_W$ are pairwise disjoint: at
$p\in T_V$ we have $V_p=T_pC\neq W_p$, so $p\notin T$ and $p\notin T_W$; at
$p\in T_W\cap T$ we would get $V_p=W_p=T_pC$, contradicting
$T_V\cap T_W=\emptyset$.

Let $\FF_V$ be the geodesic foliation generated by $V$
(Lemma~\ref{lem:geodfol}), defined on a neighborhood of $C\setminus T_V$,
and apply Theorem~\ref{thm:main} with $\FF=\FF_V$, $\psi=W$ and
$\Sigma=T_V$; the hypotheses hold since $W$ is not everywhere tangent to
$C$ and $V\wedge W\not\equiv0$. The set $S$ of \eqref{eq:S} is
$T_W\cup T_V\cup T$, and the local indices are given by
Corollary~\ref{cor:case3} at the points of $T$ (where $\FF_V$ is
holomorphic, transversal to $C$, and $W$ is transversal to $C$),
Proposition~\ref{prop:case2} at the points of $T_V$ (where $W$ is
transversal to $C$), and Proposition~\ref{prop:case1} at the points of
$T_W$ (where $\FF_V$ is holomorphic and transversal to $C$). Using
$\tfrac{k}{1-k}=\tfrac{1}{1-k}-1$ and $\tfrac{2k}{1-k}=\tfrac{2}{1-k}-2$,
Theorem~\ref{thm:main} reads
\[
\TTang(V,W)\;+\;\bigl(\Pi_V-n_V\bigr)\;+\;\bigl(2\Pi_W-2n_W\bigr)
\;=\;\deg W\;=\;C\cdot C-n_W .
\]
Substituting the tangency lemma
$\TTang(V,W)=n_V+n_W+2-2g-C\cdot C$, all tangency counts cancel and we
obtain
\[
\Pi_V+2\,\Pi_W\;=\;2\,C\cdot C+2g-2 .
\]
Exchanging the roles of $V$ and $W$ --- i.e.\ applying
Theorem~\ref{thm:main} to the geodesic foliation $\FF_W$ with reference
$V$ --- gives, by the same computation,
$2\,\Pi_V+\Pi_W=2\,C\cdot C+2g-2$. Therefore  $\Pi_V=\Pi_W=\tfrac13(2\,C\cdot C+2g-2)$.
\end{proof}

Theorem~B of the introduction is the particular case $V=T\GG|_C$, for a
foliation $\GG$ without singular points on $C$ and not leaving $C$
invariant.

\begin{corollary}\label{cor:P2}
Let $C\subset\PP^2$ be a smooth curve of degree $c$ and let $\GG$ be a
foliation of degree $d$ with no singular points on $C$ and only simple
tangencies with $C$, at the points $p_1,\dots,p_n$. Then $n=cd+c(c-1)$ and
\[
\sum_{j=1}^{n}\frac{1}{1-k(\GG,C,p_j)}\;=\;c\,(c-1),
\]
the class of $C$.
\end{corollary}

\begin{proof}
Here $\deg V=T\GG\cdot C=c(1-d)$ and $C\cdot C=c^2$, so
\eqref{eq:tangVC} gives $n=c^2-c(1-d)=cd+c(c-1)$, while
$g=\tfrac{(c-1)(c-2)}2$ gives
$\tfrac23(C\cdot C+g-1)=\tfrac23\cdot\tfrac{3c^2-3c}{2}=c(c-1)$.
\end{proof}

\begin{remark}\label{rem:degenerate}
Two extreme cases are worth recording. When $\GG$ is a pencil of lines
through a generic point, every leaf is a line and every $k$ vanishes
(Remark~\ref{rem:kvalues}), so Corollary~\ref{cor:P2} reduces to the
Pl\"ucker count of the tangent lines to $C$ from a point \cite{GHP}; for an
arbitrary $\GG$ the number of tangencies grows, but the weighted total does
not move.

At the opposite extreme, let $C$ be a compact smooth \emph{geodesic} curve.
Apply Theorem~\ref{thm:contact1} to a subbundle $W$ with only simple
tangencies with $C$, as furnished by Lemma~\ref{lem:existW} (the condition
$V\wedge W\not\equiv0$ there is not needed here). Since $C$ is a geodesic,
$\sigma_C\equiv0$; at a simple tangency $p$ Lemma~\ref{lem:simple} gives
$\Phi\bigl(\sigma_W(p)\bigr)=w'(p)\neq0$, whence $\sigma_W(p)\neq0$ and
$k(W,C,p)=\sigma_W(p)/\sigma_C(p)=\infty$. Every weight $1/(1-k)$
therefore vanishes, and the formula forces
\[
C\cdot C\;=\;1-g .
\]
\end{remark}

\section{A contact formula for two foliations along a generic
curve}\label{sec:contact2}

Let $\FF,\GG$ be singular holomorphic foliations on a neighborhood of a
compact smooth curve $C$. We say that $C$ is in \emph{general position}
with respect to $(\FF,\GG)$ if:
\begin{enumerate}
\item[(i)] $\GG$ has no singular points on $C$, and $C$ is not
$\GG$-invariant;
\item[(ii)] all tangencies of $\GG$ with $C$ are simple, and $\FF$ is
transversal to $C$ at each of them;
\item[(iii)] $\FF$ and $\GG$ are not tangent along the whole of $C$.
\end{enumerate}
Condition (iii) is automatic as soon as $\GG$ has at least one tangency
with $C$, by (ii); it is stated separately only to cover the case in which
$\GG$ is everywhere transversal to $C$.

\begin{theorem}[Contact formula for two foliations along a generic
curve]\label{thm:contact2}
Let $M$ be a complex surface endowed with a projective connection, let
$\FF,\GG$ be singular holomorphic foliations on a neighborhood of a compact
smooth curve $C$ of genus $g$, and assume $C$ is in general position with
respect to $(\FF,\GG)$. Then
\[
\sum_{p\,\in\,\TTang(\FF,\GG)\cap C}\Ind(\FF,\GG,C,p)
\;=\;\tang(\GG,C)\;-\;\frac{\iota(C)}{3}\,,
\]
where $\tang(\GG,C)$ is the number of tangencies of $\GG$ with $C$ and
$\iota(C)=C\cdot C+4(g-1)$. In particular the left hand side does not
depend on $\FF$.
\end{theorem}

\begin{proof}
Apply Theorem~\ref{thm:main} with $\psi=(T\GG\to TM)|_C$; since
$\Sing\GG\cap C=\emptyset$, $\psi$ is the inclusion of the line subbundle
$V_\GG:=T\GG|_C$, and $\pi\circ\psi\not\equiv0$ because $C$ is not
$\GG$-invariant, so $c_1(L)=T\GG\cdot C$. Here $\Sigma=\emptyset$ and the
set $S$ consists of the $n:=\tang(\GG,C)$ tangency points $p_1,\dots,p_n$
of $\GG$ with $C$, together with the finite set $T:=\TTang(\FF,\GG)\cap C$.
The two sets are disjoint, since at a
tangency of $\GG$ with $C$ the foliation $\FF$ is transversal to $C$, hence
transversal to $\GG$. At the tangencies of $\GG$ with $C$,
Proposition~\ref{prop:case1} applies and gives the index
$2k_j/(1-k_j)$ with $k_j=k(\GG,C,p_j)$ (Remark~\ref{rem:kvalues}). Hence
\[
\sum_{p\in T}\Ind(\FF,\GG,C,p)
\;+\;\sum_{j=1}^{n}\Bigl(\frac{2}{1-k_j}-2\Bigr)
\;=\;T\GG\cdot C .
\]
By Theorem~\ref{thm:contact1},
$\sum_j 1/(1-k_j)=\tfrac23(C\cdot C+g-1)$, and by \eqref{eq:tangVC},
$T\GG\cdot C=C\cdot C-n$. Therefore
\[
\sum_{p\in T}\Ind
= C\cdot C-n-\tfrac43\,(C\cdot C+g-1)+2n
= n-\tfrac13\,\bigl(C\cdot C+4(g-1)\bigr). \qedhere
\]
\end{proof}

The individual terms of the left hand side are given by
Proposition~\ref{prop:case4}, and depend on $\FF$ through the geodesic
defect $\nu$ of its leaves; only their sum does not.

When $C$ is not a geodesic, $\iota(C)$ is the number of inflection points
of $C$ (Proposition~\ref{prop:inflections}) and the formula reads:
\emph{whatever the foliation $\FF$, the total index of its tangencies with
$\GG$ along $C$ equals the number of tangencies of $\GG$ with $C$, minus one third of the number of inflection points of $C$.}

\begin{corollary}\label{cor:contact2P2}
In $\PP^2$, with $\deg C=c$ and $\deg\GG=d$,
\[
\sum_{p\,\in\,\TTang(\FF,\GG)\cap C}\Ind(\FF,\GG,C,p)\;=\;c\,(d+1).
\]
\end{corollary}

\begin{proof}
$n=cd+c(c-1)$ and $\iota(C)=3c(c-2)$, so
$n-\iota(C)/3=cd+c(c-1)-c(c-2)=c(d+1)$.
\end{proof}

\end{document}